\documentclass[11pt, letterpaper]{amsart}
\usepackage[utf8]{inputenc}
\usepackage{setspace}
\usepackage{mathrsfs}
\usepackage{amssymb}
\usepackage{latexsym}
\usepackage{amsfonts}
\usepackage{amsmath}
\usepackage{eucal}
\usepackage{bm}
\usepackage{bbm}
\usepackage{graphicx}
\usepackage[english]{varioref}
\usepackage[nice]{nicefrac}
\usepackage[all]{xy}
\usepackage{amsthm}
\usepackage{amssymb,amsthm,upref,amscd}

\def\op{\operatorname}

\def\mmod{\kern-1pt\operatorname{-mod}}

\newtheorem{theorem}{Theorem}[section]
\newtheorem{lemma}[theorem]{Lemma}

\newtheorem{remark}[theorem]{Remark}
\newtheorem{proposition}[theorem]{Proposition}
\newtheorem{corollary}[theorem]{Corollary}

\theoremstyle{proposition}

\numberwithin{equation}{section}

\begin{document}

\title[Classification of simple modules of  $SL_2(\bar{\mathbb{F}}_p)$  ]{A note  on the classification of simple   $SL_2(\bar{\mathbb{F}}_p)$-modules  admitting $\bf T$-stable lines in cross characteristic}

\author{Junbin Dong}
\address{Institute of Mathematical Sciences, ShanghaiTech University, 393 Middle Huaxia Road, Pudong, Shanghai 201210, China.}
\email{dongjunbin@shanghaitech.edu.cn}




\begin{abstract}
 Let $\bf T$ be the group of  diagonal matrices in $SL_2(\bar{\mathbb{F}}_p)$, where $p$ is a prime number.  Let  $\Bbbk$ be an algebraically closed field  of characteristic not equal to $2$ and $p$.  We classify all the irreducible $\Bbbk$-representations of $SL_2(\bar{\mathbb{F}}_p)$ that admit $\bf T$-stable lines.
\end{abstract}

\maketitle

\section{Introduction}

Let  ${\bf G}= SL_2(\bar{\mathbb{F}}_p)$, where $\mathbb{F}_p$ denotes the finite field of $p$ elements and $p$ is a  prime number.
Let  $\Bbbk$ be an algebraically closed field whose characteristic $r$ is different from $2$ and $p$.
According to  a result of Borel and Tits \cite[Theorem 10.3 and Corollary 10.4]{BT}, we know that except the trivial representation, all other irreducible representations of $SL_2(\bar{\mathbb{F}}_p)$ are infinite dimensional. N.H. Xi studied the abstract representations of $SL_2(\bar{\mathbb{F}}_p)$ by taking the direct limit of the finite-dimensional representations of $SL_2(\mathbb{F}_{p^m})$ and he got many interesting results in \cite[Section 6]{Xi}.  A natural question concerns  the  classification of simple   $SL_2(\bar{\mathbb{F}}_p)$-modules. However this problem  is  by no means easy to tackle.  This can be understood from the fact that  $\mathfrak{sl}_2(\mathbb{C})$ is the only complex semisimple Lie algebra whose simple modules are described (see \cite{M}).

Let  ${\bf T}$  be the diagonal matrices and ${\bf U}$ be the strictly upper unitriangular matrices in $SL_2(\bar{\mathbb{F}}_p)$. Let ${\bf B}$ be the Borel subgroup generated by ${\bf T}$ and ${\bf U}$, which is the upper triangular matrices in $SL_2(\bar{\mathbb{F}}_p)$.
There are two natural isomorphisms
$$h:   \bar{\mathbb{F}}^*_p\rightarrow{\bf T}, \   h(c)=\begin{pmatrix}c &0 \\0&c^{-1}\end{pmatrix}; \ \ \ \
\varepsilon : \bar{\mathbb{F}}_p\rightarrow{\bf U}, \ \  \varepsilon(x)= \begin{pmatrix}1&x\\0&1\end{pmatrix} $$
which satisfies $h(c)\varepsilon(x)h(c)^{-1}= \varepsilon(c^2x)$. The simple root
$\alpha: {\bf T}\rightarrow  \bar{\mathbb{F}}^*_p$ is given by $\alpha(h(c))= c^2$.
Let ${\bf N}$ be the normalizer of ${\bf T}$ in ${\bf G}$ and  $W= {\bf N}/{\bf T}$ be the  Weyl group.  We set
$s=\begin{pmatrix}0 &1\\  -1 &0\end{pmatrix}$, which is the simple reflection of $W$.
One has that
\begin{align*}
s \varepsilon(x) s=\displaystyle \varepsilon(-x^{-1}) s h(-x) \varepsilon(-x^{-1}).
\end{align*}
Let $M$ be  any representation of $SL_2(\bar{\mathbb{F}}_p)$. For any character $\theta:  {\bf T}  \rightarrow  \Bbbk^*$, we set
$$M_{\theta}= \{m\in M \mid  t m= \theta(t) m, \forall  t\in {\bf T}\}.$$
Denote $M^{\bf T}= \displaystyle \bigoplus_{\theta \in  \widehat{\bf T}} M_{\theta}$, where $ \widehat{\bf T}$ is the character group of $ \bf T$.
We say that a representation $M$  of   $SL_2(\bar{\mathbb{F}}_p)$  admits $\bf T$-stable lines if $M^{\bf T}\ne 0$.

In \cite{D}, we classify all the  simple $SL_2(\bar{\mathbb{F}}_p)$-modules that admit $\bf T$-stable lines under the assumption that  $\text{char}\ \Bbbk=0$.  For completeness, we list  all the simple modules. Let $\theta \in \widehat{\bf T}$  and it  can be regarded as a character of ${\bf B}$ by the homomorphism ${\bf B}\rightarrow{\bf T}$. Let ${\Bbbk}_\theta$ be the corresponding ${\bf B}$-module. We consider the induced module $\mathbb{M}(\theta)=\Bbbk{\bf G}\otimes_{\Bbbk{\bf B}}{\Bbbk}_\theta$. Let ${\bf 1}_{\theta}$ be a fixed nonzero element in ${\Bbbk}_\theta$. We abbreviate $x{\bf 1}_{\theta}:=x\otimes{\bf 1}_{\theta}\in\mathbb{M}(\theta)$ for $x\in {\bf G}$.
When $\theta$ is trivial,  we set $\eta_s= (1-s) {\bf 1}_{\text{tr}}$ and thus  $\op{St}=  \Bbbk {\bf U} \eta_{s}$ is the Steinberg module.
Moreover, $\mathbb{M}(\op{tr})/\op{St}$ is the trivial module $ \Bbbk_{\text{tr}}$ of $SL_2(\bar{\mathbb{F}}_p)$.
By  \cite[Theorem 3.1]{CD2},  $ \Bbbk_{\op{tr}}$, $\op{St}$, $\mathbb{M}(\theta)$  are simple $SL_2(\bar{\mathbb{F}}_p)$-modules,  where $\theta \in \widehat{\bf T}$ is nontrivial. The main result of this paper is as follows.

\begin{theorem} \label{mainthm}
Let  $\Bbbk$ be an algebraically closed field  of characteristic not equal to $2$ and $p$.  Then  all  simple $SL_2(\bar{\mathbb{F}}_p)$-modules
over $\Bbbk$ admitting $\bf T$-stable lines are precisely  $ \Bbbk_{\op{tr}}$, $\op{St}$,  and $\mathbb{M}(\theta)$ for all nontrivial characters  $\theta \in \widehat{\bf T}$.
\end{theorem}

According to \cite[Proposition 4.3]{D1} and \cite[Proposition 4.5]{D1}, one direct summand of $\text{St} \otimes \text{St}$ (denoted by $V_-$) is finitely generated and thus  has simple quotients. However, any simple quotient of $V_-$ does not admit $\bf T$-stable lines. The simple quotients of  $V_-$ yield some new infinite-dimensional simple modules of $SL_2(\bar{\mathbb{F}}_p)$  that have not been previously observed.

Now let $M$ be  a representation of  $SL_2(\bar{\mathbb{F}}_p)$ such that $M_{\theta} \ne 0$ for some $\theta\in \widehat{\bf T}$. Then using Frobenius reciprocity, we have
$$\op{Hom}_{\Bbbk {\bf G}}(\op{Ind}^{\bf G}_{\bf T} \Bbbk_{\theta}, M) \cong \op{Hom}_{\Bbbk {\bf T}}( \Bbbk_{\theta}, M) \cong M_{\theta} \ne 0.$$
In order to prove Theorem \ref{mainthm}, it is enough to
study the simple quotients of the induced module $\op{Ind}_{\bf T}^{\bf G} \Bbbk_\theta$ for each $\theta \in \widehat{\bf T}$.
Therefore, the rest of the paper is organized as follows:  in Section 2,  we consider the simple quotients of the induced module $\op{Ind}_{\bf T}^{\bf G} \Bbbk_\theta$ where $\theta$ is trivial. Section 3 presents a key lemma in preparation for the nontrivial case, which is studied in   Section 4. In Section 5, we provide a supplementary result concerning the linear  independence of the translates of  nontrivial  multiplicative characters of finite fields.

\section{The trivial character case}
Firstly, we introduce some notations.
For any finite subset $X$ of  ${\bf G}= SL_2(\bar{\mathbb{F}}_p)$, let $\overline{X}:=\displaystyle \sum_{x\in X}x \in\Bbbk{\bf G}$.   We set
$U_{p^n}= \{\varepsilon(x)\mid x\in \mathbb{F}_{p^n}\}.$
In this section, we  deal with  the case where $\theta$ is trivial.
Noting that $\op{char} \Bbbk \ne 2$  and ${\bf N} / {\bf T} \cong W=\{e, s\} $, it is easy to see  that
$$\op{Ind}_{\bf T} ^{\bf G} \Bbbk_{\text{tr}}\cong \op{Ind}_{\bf N} ^{\bf G} \Bbbk_{+} \oplus \op{Ind}_{\bf N} ^{\bf G} \Bbbk_{-},$$
where $\Bbbk_{+}$ is the trivial representation of ${\bf N}$ and $\Bbbk_{-}$ is the sign representation of $W$, which also can be regarded as a representation of  ${\bf N}$.  Let  ${\bf 1}_{\op{tr}}$   be a fixed nonzero element in $\Bbbk_{\op{tr}}$. Denote  ${\bf 1}_{+}=\frac{1}{2}(e+s) {\bf 1}_{\op{tr}}$ and ${\bf 1}_{-}=\frac{1}{2}(e-s) {\bf 1}_{\op{tr}}$.
We abbreviate $x{\bf 1}_{+}:=x\otimes{\bf 1}_{+}\in \op{Ind}_{\bf N} ^{\bf G} \Bbbk_{+}$ and
$x{\bf 1}_{-}:=x\otimes{\bf 1}_{-}\in \op{Ind}_{\bf N} ^{\bf G} \Bbbk_{-}$ for any $x\in {\bf G}$.

\begin{lemma} \label{trReplem}
Let $M$ be a  $\Bbbk {\bf G}$-submodule of $ \op{Ind}_{\bf T} ^{\bf G} \Bbbk_{\op{tr}}$. If there exists $n\in \mathbb{N}$ such that
$$\xi=\sum_{z\in \mathbb{F}^*_{p^n} }  a_z   \overline{U_{p^n}} s  \varepsilon(z) {\bf 1}_{\op{tr}} \in M, $$
where $a_z\in \Bbbk$ and $A:= \displaystyle \sum_{z\in \mathbb{F}^*_{p^n} }  a_z  \ne 0$.
Then we have $( e+ s){\bf 1}_{\op{tr}} \in M$.
\end{lemma}

\begin{proof}   Choose $\mathfrak{D}_{p^n}\subset {\bf T}$ such that $\alpha:\mathfrak{D}_{p^n}\rightarrow  \mathbb{F}^*_{p^n}$ is a bijection. It is easy to check that
\begin{align}\label{eq2.1}
\overline{\mathfrak{D}_{p^n}}\ \xi =  A \overline{U_{p^n}} s \overline{U^*_{p^n}}{\bf 1}_{\op{tr}}  \in M.
\end{align}
which implies that  $\displaystyle \overline{U_{p^n}} s \overline{U^*_{p^n}}{\bf 1}_{\op{tr}}\in M$ since $A\ne 0$. Let $V$  be the space spanned by$\{  \overline{U_{p^n} } {\bf 1}_{\op{tr}}, \overline{U_{p^n}} s  \varepsilon(z){\bf 1}_{\op{tr}}\mid  z\in \mathbb{F}_{p^n}\}$.
Let $\omega$ be the generator of the cyclic group $\mathbb{F}^*_{p^n}$.  Denote by $\mathcal{A}:  V \rightarrow V$ the linear map induced  by the action of $h(\omega^{\frac{1}{2}})$ on $V$.  Then the minimal polynomial of $\mathcal{A}$ is $t^{p^n-1}-1$.   The eigenspace corresponding to the eigenvalue $1$ is spanned by $ \overline{U_{p^n}} {\bf 1}_{\op{tr}}$,   $\overline{U_{p^n}} s{\bf 1}_{\op{tr}}$  and $\overline{U_{p^n}} s \overline{U^*_{p^n}}{\bf 1}_{\op{tr}}$. It is easy to see that  $\overline{U_{p^n}} s \xi$ has the following form
\begin{align}\label{eq2.2}
\overline{U_{p^n}} s    \xi= A( \overline{U_{p^n}} {\bf 1}_{\op{tr}} +\overline{U_{p^n}} s{\bf 1}_{\op{tr}}  )+\sum_{z\in \mathbb{F}^*_{p^n} }  b_z   \overline{U_{p^n}} s  \varepsilon(z) {\bf 1}_{\op{tr}},
\end{align}
where $b_z\in \Bbbk$.  When $r \nmid p^n-1$, then the linear map $\mathcal{A}$ is diagonalizable.
 Thus using (\ref{eq2.2}),  it is easy to see that $$A(\overline{U_{p^n}} {\bf 1}_{\op{tr}} +\overline{U_{p^n}} s{\bf 1}_{\op{tr}}) + b \overline{U_{p^n}} s \overline{U^*_{p^n}}{\bf 1}_{\op{tr}} \in M$$
for some $b \in \Bbbk$.  Noting that $\overline{U_{p^n}} s \overline{U^*_{p^n}}{\bf 1}_{\op{tr}} \in M $ by (\ref{eq2.1}), we get  $\overline{U_{p^n}} {\bf 1}_{\op{tr}}  +\overline{U_{p^n}} s{\bf 1}_{\op{tr}}  \in M$. Thus by   {\cite[Lemma 3.6]{CD2}}, we have $( e+ s){\bf 1}_{\op{tr}} \in M$.

Now we consider the case $r| p^n-1$.  Write $p^n-1= ml$, where $m$ is a power of $r$ and $(l,r)=1$. Then the Jordan canonical form of $\mathcal{A}$ is $$J_1(\lambda_1) \dot{+} J_1(\lambda_1) \dot{+} J_{m}(\lambda_1) \dot{+} J_{m}(\lambda_2) \dot{+} \cdots \dot{+} J_{m}(\lambda_l),$$
where $\lambda_1=1, \lambda_2, \cdots, \lambda_l$ are the roots of the equation $t^l-1=0$.  Let $\mathbf{u}_1, \mathbf{u}_2, \cdots, \mathbf{u}_m$ be a Jordan basis corresponding to the block $J_m(\lambda)$ ($\lambda \ne 1$), where $(\mathcal{A}- \lambda \mathcal{E}) \mathbf{u}_{i+1}= \mathbf{u}_i$. Set $ \mathbf{w}_i = \overline{U_{p^n}}  s  \mathbf{u}_i \in V$. Note that $\mathcal{A} \overline{U_{p^n}}  s  \mathbf{u}_i = \overline{U_{p^n}}  s \mathcal{A}^{-1} \mathbf{u}_i $.  Then we have
$$ \mathcal{A} \mathbf{w}_{i+1}=  \overline{U_{p^n}}  s  \mathcal{A}^{-1} \mathbf{u}_{i+1} = \lambda^{-1}\overline{U_{p^n}}  s (  \mathbf{u}_{i+1} -  \mathcal{A}^{-1}\mathbf{u}_{i} )=\lambda^{-1}( \mathbf{w}_{i+1}- \mathcal{A} \mathbf{w}_{i}).$$
Thus  $(  \mathcal{A}- \lambda^{-1}  \mathcal{E})  \mathbf{w}_{i+1}= - \lambda^{-1}\mathcal{A} \mathbf{w}_{i}$.
Therefore we have 
\begin{align}\label{eq2.7}
\mathbf{w}_{i}=\overline{U_{p^n}}  s  \mathbf{u}_i \in \ker{(\mathcal{A}- \lambda^{-1}  \mathcal{E})^{i} } 
\end{align} 
by induction on $i$. In particular, we see that   
\begin{align}\label{eq2.8}
\overline{U_{p^n}} s:  \ker{(\mathcal{A}- \lambda  \mathcal{E})^{m}} \rightarrow  \ker{(\mathcal{A}- \lambda^{-1} \mathcal{E})^{m}}
\end{align}
is a linear operator.

Let $\mathbf{v}_1, \mathbf{v}_2, \cdots, \mathbf{v}_m$ be a Jordan basis corresponding to the block $J_m(1)$, where $(\mathcal{A}- \mathcal{E}) \mathbf{v}_{i+1}= \mathbf{v}_i$ and $\mathbf{v}_1= \overline{U_{p^n}} s \overline{U^*_{p^n}}{\bf 1}_{\op{tr}}$.    Let $ \mathbf{v}'_i = \overline{U_{p^n}}  s  \mathbf{v}_i \in V$, which is in $ \ker{(\mathcal{A}- \mathcal{E})^{i} }$. Denote $\mathbf{x}= \overline{U_{p^n}} {\bf 1}_{\op{tr}}$ and $\mathbf{y}= \overline{U_{p^n}}s {\bf 1}_{\op{tr}}$. Thus $ \mathbf{v}'_k$ is a linear combination of $\mathbf{x}, \mathbf{y},  \mathbf{v}_1, \mathbf{v}_2, \cdots, \mathbf{v}_k$ for any $k=1,2,\cdots, m$.
Taking the direct summand  of  $\xi $ in the Jordan block $J_m(1)$, we get an element
$$\xi'= a_1 \mathbf{v}_1 +a_2 \mathbf{v}_2  \cdots  + a_k\mathbf{v}_k \in M, $$
where $a_1, a_2, \cdots, a_k \in \Bbbk$ and $a_k\ne 0$.    Thus we have $\mathbf{v}_k \in M$.  In (\ref{eq2.2}), take the direct summand  of  $\overline{U_{p^n}} s \xi$ in  the  Jordan block $J_1(\lambda_1) \dot{+} J_1(\lambda_1) \dot{+} J_m(1)$. Using (\ref{eq2.7}) and  (\ref{eq2.8}),  we get  an element
$$\xi''= A(\mathbf{x}+ \mathbf{y})+b_1  \mathbf{v}_1 + \cdots  + b_{k}\mathbf{v}_{k} \in M.$$
Thus we have $\mathbf{x}+ \mathbf{y}=\overline{U_{p^n}} {\bf 1}_{\op{tr}} +\overline{U_{p^n}} s{\bf 1}_{\op{tr}} \in M$.
Using {\cite[Lemma 3.6]{CD2}}, we get $( e+ s){\bf 1}_{\op{tr}} \in M$. The lemma is proved.
\end{proof}

Now we consider the simple quotient of $\op{Ind}_{\bf N} ^{\bf G} \Bbbk_{+}$. Let
$$\mathbb{M}_{+}=\{\sum_{g\in{\bf G}}a_g g {\bf 1}_{+} \mid\sum_{g\in{\bf G}}a_g=0\},$$
which is a $\Bbbk {\bf G}$-submodule of $\op{Ind}_{\bf N} ^{\bf G} \Bbbk_{+}$.
It is easy to see that $\op{Ind}_{\bf N} ^{\bf G} \Bbbk_{+} /\mathbb{M}_{+}$ is the trivial $\Bbbk {\bf G}$-module.
We show that $\mathbb{M}_{+}$ is the unique maximal submodule of $\op{Ind}_{\bf N} ^{\bf G} \Bbbk_{+}$.
Let $\xi \notin \mathbb{M}_{+}$ with the following expression
$$\xi=  \sum_{x, y\in \mathbb{F}_{p^m} } a_{x,y} \varepsilon(x)s \varepsilon(y){\bf 1}_{+}, \quad  \text{where} \  \sum_{x, y\in \mathbb{F}_{p^m} } a_{x,y}\ne 0.$$
Let $n>m$ with $m | n$ and  fix $u\in \mathbb{F}_{p^n} \setminus \mathbb{F}_{p^m}$. Let
$\eta= s \varepsilon(u) \xi $, which has the form
$$\eta = \displaystyle \sum_{x, y\in   \mathbb{F}_{p^n} } a_{x,y} \varepsilon(-(u+x)^{-1})s \varepsilon((u+x)^2y-(u+x)){\bf 1}_{+}. $$
Then we have
$$ \overline{U_{p^n}}\eta=\sum_{z\in \mathbb{F}^*_{p^n} } c_z   \overline{U_{p^n}} s  \varepsilon(z){\bf 1}_{+}, $$
where $\displaystyle  \sum_{z\in \mathbb{F}^*_{p^n} } c_z=  \sum_{x, y\in \mathbb{F}_{p^m}} a_{x,y} \ne 0$.
Note that  ${\bf 1}_{+}=\frac{1}{2}(e+s) {\bf 1}_{\op{tr}}$ and
$$  \overline{U_{p^n}} s  \varepsilon(z) s {\bf 1}_{\op{tr}} =\overline{U_{p^n}} s  \varepsilon(-z)  {\bf 1}_{\op{tr}}.$$
Therefore we have
$$  \overline{U_{p^n}}\eta=\frac{1}{2}\sum_{z\in \mathbb{F}^*_{p^n} } c_z  \overline{U_{p^n}} s  \varepsilon(z){\bf 1}_{\op{tr}}+ \frac{1}{2}\sum_{z\in \mathbb{F}^*_{p^n} }  c_z   \overline{U_{p^n}} s  \varepsilon(-z){\bf 1}_{\op{tr}}.$$
Thus by Lemma \ref{trReplem}, we see that ${\bf 1}_{+}\in \Bbbk {\bf G}\xi$ and
thus for any element $\xi \notin \mathbb{M}_{+}$,  we have $\Bbbk {\bf G}\xi =\op{Ind}_{\bf N} ^{\bf G} \Bbbk_{+}$. Therefore, the trivial module is the unique simple quotient module of $\op{Ind}_{\bf N} ^{\bf G} \Bbbk_{+}$.

Now we consider the simple  quotient modules of $\op{Ind}_{\bf N} ^{\bf G} \Bbbk_{-}$.  Let $\text{St}= \Bbbk {\bf U} \eta_{s}$ be the Steinberg module, where $\eta_s= (1-s) {\bf 1}_{\text{tr}}$. Then we have $s\eta_s =-\eta_s$ and $s \varepsilon(x)\eta_s =(\varepsilon(-x^{-1})-1) \eta_s$ for any $x\in \bar{\mathbb{F}}^*_p$. We show that $\op{St}$ is the unique simple module of $\op{Ind}_{\bf N} ^{\bf G} \Bbbk_{-}$.
For convenience, we denote
$$\Lambda(z)= (s \varepsilon(z) +1- \varepsilon(-{z}^{-1})) {\bf 1}_{-}$$
for each $z\in \bar{\mathbb{F}}^*_p$. Then we have $h(c)\Lambda(z)= \Lambda(c^{-2}z)$ for any $c\in \bar{\mathbb{F}}^*_p$. Moreover, it is easy to get the following lemma.
\begin{lemma} \label{easylemma}
The following identities hold:
  $$\varepsilon(z)\Lambda(z^{-1})=-\Lambda(-z^{-1}) ,\quad s\Lambda(z)=-\Lambda(-z^{-1})$$
  for any $z\in \bar{\mathbb{F}}^*_p$ and
  $$s\varepsilon(x)\Lambda(y)=\varepsilon(-x^{-1})\Lambda(x(xy-1))+\Lambda(x)-\Lambda(y^{-1}(xy-1)),$$
 where $x,y\in \bar{\mathbb{F}}^*_p$ and $xy\ne 1$.
\end{lemma}

Let $\mathbb{M}_{-}$ be the submodule of $\op{Ind}_{\bf N} ^{\bf G} \Bbbk_{-}$ which is generated by $\Lambda(z), z\in \bar{\mathbb{F}}^*_p$. Thus $\op{Ind}_{\bf N} ^{\bf G} \Bbbk_{-}/ \mathbb{M}_{-} $ is isomorphic to the Steinberg module.
Let $\zeta\in \op{Ind}_{\bf N} ^{\bf G} \Bbbk_{-}$ which is not in $ \mathbb{M}_{-} $. According to Lemma \ref{easylemma}, then $\zeta$ has the following expression
$$\zeta= \sum_{xy\ne 1}a_{x,y} \varepsilon(x)\Lambda(y) {\bf 1}_{-} + \sum_{z\in \bar{\mathbb{F}}^*_p} b_{z} \Lambda(z){\bf 1}_{-}+ \sum_{u\in \bar{\mathbb{F}}_p} c_u\varepsilon(u) {\bf 1}_{-}  ,$$
where $c_u \ne 0$ for some $u$. There exists an integer $n\in \mathbb{N}$ such that $x, y,z,u\in \mathbb{F}_{p^n}$ whenever  $a_{x,y}\ne 0, b_{z}\ne 0$ and $c_u\ne 0$. Without loss of generality, we  assume that $c_{0}\ne 0$. Moreover, we assume that $\displaystyle \sum_{u\in \bar{\mathbb{F}}_p} c_u \ne 0$. Otherwise, we  consider $s \zeta$ instead of $\zeta$. Indeed, if we write
$$s\zeta=\sum_{x y\ne 1}a'_{x,y} \varepsilon(x)\Lambda(y) {\bf 1}_{-} + \sum_{z\in \bar{\mathbb{F}}^*_p} b'_{z} \Lambda(z){\bf 1}_{-}+ \sum_{u\in \bar{\mathbb{F}}_p} c'_{u} \varepsilon(u){\bf 1}_{-},$$
it is easy to see that $\displaystyle \sum_{u\in \bar{\mathbb{F}}_p} c'_{u}=-c_0$ which is nonzero using Lemma \ref{easylemma}.
For the convenience of later discussion, we denote by
$\displaystyle C= \sum_{u\in \bar{\mathbb{F}}_p} c_u$.
Note that $C$ is nonzero by our assumption. Thus we have
\begin{align}\label{eq2.3}
\overline{U_{p^n}} \zeta= C\overline{U_{p^n}} {\bf 1}_{-} +\sum_{z\in \mathbb{F}^*_{p^n}} d_z \overline{U_{p^n}}s\varepsilon(z){\bf 1}_{-},
\end{align}
where $d_z\in \Bbbk$.  It is easy to check that
\begin{align}\label{eq2.4}
\overline{U_{p^n}}s\varepsilon(z){\bf 1}_{-}= - \overline{U_{p^n}}s\varepsilon(z)s{\bf 1}_{-}=-\overline{U_{p^n}}s\varepsilon(-z){\bf 1}_{-}.
\end{align}
Let $t_0\in \bar{\mathbb{F}}^*_p$ such that $t^2_0=-1$. Using (\ref{eq2.3}) and (\ref{eq2.4}),  we get
$$ h(t_0)\overline{U_{p^m}} \zeta +\overline{U_{p^m}} \zeta =2C\overline{U_{p^n}} {\bf 1}_{-},$$
which implies that $\overline{U_{p^n}}{\bf 1}_{-} \in \Bbbk {\bf G}\zeta$.  Thus  we get ${\bf 1}_{-} \in \Bbbk {\bf G}\zeta $ by {\cite[Lemma 3.6]{CD2}}. So for any element $\zeta \notin \mathbb{M}_{-}$,  we see that $\Bbbk {\bf G}\zeta =\op{Ind}_{\bf N} ^{\bf G} \Bbbk_{-}$ and thus the Steinberg module is the unique simple quotient module of $\op{Ind}_{\bf N} ^{\bf G} \Bbbk_{-}$.

\section{A key lemma }

Following  the proof strategy in the paper \cite{D}, we find that  \cite[Corollary 3.3]{D} plays a key role.
The original  proof of   \cite[Corollary 3.3]{D}  in characteristic zero case relies on the main results established in \cite[Theorem 1.1]{ESS}.
In this section, we prove the following lemma, which is a generalization of \cite[Corollary 3.3]{D}.

\begin{lemma}\label{keylemma}
Let  $\Bbbk$ be an algebraically closed field with characteristic $r \ne p$.   Let  $\lambda : \bar{\mathbb{F}}^*_p \rightarrow \Bbbk^*$ be a nontrivial  multiplicative character and $u_1, u_2, \dots, u_n\in \bar{\mathbb{F}}^*_p$,  which are different from each other.
Given $a_1, a_2,\dots, a_n \in \Bbbk^*$, then there exists infinitely many elements $x\in  \bar{\mathbb{F}}^*_p$ such that
$$a_1\lambda (x+u_1)+ a_2 \lambda(x+u_2)+\dots + a_n\lambda(x+u_n)\ne 0.$$
\end{lemma}

We will use the orthogonality relations of characters for finite groups to prove Lemma \ref{keylemma}. The following proposition  is a well-known fact.

\begin{proposition}  \label{Orth}
Let $ \varphi,  \psi$ be two multiplicative characters of  $\mathbb{F}^*_{p^m}$ over $\Bbbk$.
Then
$$ \sum_{z\in {\mathbb{F}}^*_{p^m}} \varphi(z) \psi(z^{-1}) =(p^m-1)\delta_{\varphi,  \psi}.$$
In particular, one has that $ \displaystyle  \sum_{z\in {\mathbb{F}}^*_{p^m}} \chi(z) =0$ for any nontrivial multiplicative character $\chi$ of  $\mathbb{F}^*_{p^m}$.

\end{proposition}

\begin{lemma} \label{infiniteX}
 Let $p,r$ be different  prime numbers such that $p \not\equiv 1  \pmod r $.
One has that

\noindent (i) The set $X=\{ m\in \mathbb{N} \mid r \nmid p^m-1\}$ is infinite.

\noindent (ii)  For any $k\in \mathbb{N}$, there exist  $m, l\geq k$ such that $p^m \not\equiv p^{l}  \pmod r$.

\end{lemma}

\begin{proof} (i) Let $d$ be the order of $p$ in $\mathbb{F}_{r}= \mathbb{Z}/r\mathbb{Z}$, i.e., $p^d \equiv 1  \pmod r$.
Noting that  $p \not\equiv 1 \pmod r$, we see that $d\geq 2$.
Thus $X= \{ m\in \mathbb{N} \mid d \nmid m\}$, which is infinite and (i) is proved.

\noindent (ii)  Since $X$ is infinite,  for any $k\in \mathbb{N}$, let $m \in X$ and $m\geq k$.  We set $l=2m$, then
 $$p^{2m}-p^m=p^{m}(p^{m}-1)  \not\equiv 0  \pmod r,$$
 which proves (ii).
\end{proof}

\begin{proof}  [{Proof of Lemma \ref{keylemma}}]
Since  $\lambda$ is nontrivial,  we set  $\lambda (0)=0$ for convenience. Let
 $\displaystyle S(x)= \sum_{i=1}^n  a_i \lambda (x+u_i).$
 Denote $Z=\{x\in \bar{\mathbb{F}}^*_p\mid  S(x)\ne 0\}$. Suppose that the set $Z$ is finite. There exists an integer $m$, such that  $\mathbb{F}_{p^m} \supseteq Z\cup \{u_1, u_2, \cdots, u_n\}$ and $\lambda|_{ {\mathbb{F}}^*_{p^m}}$ is a nontrivial  multiplicative character of
${\mathbb{F}}^*_{p^m}$.

Firstly we consider the case $p \not \equiv 1  \pmod r $.  Using Proposition \ref{Orth}, for any $u\in  \mathbb{F}^*_{p^m}$, we see that
 $$\sum_{z\in  \mathbb{F}^*_{p^m}} \lambda(z+u) \lambda(z^{-1})=\sum_{z\in  \mathbb{F}^*_{p^m}} \lambda(1+uz^{-1})=-1.  $$
 Therefore it is not difficult to see that
 \begin{align}\label{eq3.1}
\sum_{x\in  \mathbb{F}_{p^m}\setminus  \{-u_1\} } S(x) \lambda((x+u_1)^{-1})=p^m a_1- \sum_{i=1}^n a_i.
 \end{align}
 By  Lemma \ref{infiniteX}, there exists an integer $l$  such that $m \mid l$ and $p^m \not\equiv p^l  \pmod r$.
  Since  $\mathbb{F}_{p^l} \supseteq \mathbb{F}_{p^m} \supseteq Z\cup \{u_1, u_2, \cdots, u_n\}$,  we see that
  $$\sum_{x\in  \mathbb{F}_{p^m}\setminus  \{-u_1\} } S(x) \lambda((x+u_1)^{-1})= \sum_{x\in  \mathbb{F}_{p^l}\setminus  \{-u_1\} } S(x) \lambda((x+u_1)^{-1}),$$
  which implies that $p^m a_1 =p^l a_1$ by (\ref{eq3.1}). Noting that $p^m \not\equiv p^l  \pmod r$, we get a contradiction.
Thus Lemma \ref{keylemma} is proved in the case $p \not \equiv 1  \pmod r $.

Now we consider the case $p  \equiv 1  \pmod r $.  Note that $S(x) =0$ for any $x\in \overline{\mathbb{F}}_p\setminus\mathbb{F}_{p^m} $ by the assumption.  So we get
\begin{align}\label{eq3.2}
S(xy)= \sum_{i=1}^n  a_i \lambda (xy+u_i)=0
\end{align}
for any $x\in \overline{\mathbb{F}}_p\setminus\mathbb{F}_{p^m} $ and $y\in \mathbb{F}^*_{p^m}$, which implies that $\displaystyle \sum_{i=1}^n  a_i \lambda (x+u_i y^{-1})=0$.
Therefore we have
$$ \sum_{y\in \mathbb{F}^*_{p^m}} \sum_{i=1}^n  a_i \lambda (x+u_i y^{-1})=(\sum_{i=1}^n a_i) \sum_{z\in \mathbb{F}^*_{p^m}} \lambda(x+z)$$
for any $x\in \overline{\mathbb{F}}_p\setminus\mathbb{F}_{p^m}$.  If $\displaystyle \sum_{i=1}^n a_i \ne 0$, then $\displaystyle \sum_{z\in \mathbb{F}^*_{p^m}} \lambda(x+z)=0$ for any $x\in \overline{\mathbb{F}}_p\setminus\mathbb{F}_{p^m}$.
If $\displaystyle \sum_{i=1}^n a_i =0$, then we consider $S(x-u_1)$ instead of $S(x)$ and by (\ref{eq3.2}), we get
$$ a_1 \lambda(x)+ \sum_{i=2}^n a_i \lambda(x+ (u_i-u_1) y^{-1})=0$$
for any $x\in \overline{\mathbb{F}}_p\setminus\mathbb{F}_{p^m} $ and $y\in \mathbb{F}^*_{p^m}$.
Noting that  $p  \equiv 1  \pmod r $, thus we have
$$ \sum_{y\in  \mathbb{F}^*_{p^m} } \big(a_1 \lambda(x)+ \sum_{i=2}^n a_i \lambda(x+ (u_i-u_1) y^{-1}) \big)=(\sum_{i=2}^n a_i) \sum_{z\in \mathbb{F}^*_{p^m}} \lambda(x+z)=0$$
by Proposition \ref{Orth}. Since $\displaystyle \sum_{i=2}^n a_i \ne 0$,  we also have   $\displaystyle \sum_{z\in \mathbb{F}^*_{p^m}} \lambda(x+z)=0$ for any $x\in \overline{\mathbb{F}}_p\setminus\mathbb{F}_{p^m}$.

Denote  $R_m(x)= \displaystyle \sum_{z\in \mathbb{F}_{p^m}} \lambda(x+z)$ and thus
\[
R_m(x)=
\begin{cases}
0, & x \in   \mathbb{F}_{p^m}\\
\lambda(x), & x \in \overline{\mathbb{F}}_p\setminus\mathbb{F}_{p^m}.
\end{cases}
\]
Let $l\in \mathbb{N}$ such that $m | l$ and $l>m$, we also have
\[
R_l(x)= \sum_{z\in \mathbb{F}_{p^l}} \lambda(x+z)=
\begin{cases}
0, & x \in   \mathbb{F}_{p^l}\\
\lambda(x), & x \in \overline{\mathbb{F}}_p\setminus\mathbb{F}_{p^l}.
\end{cases}
\]
Let $\{0, z_1, z_2, \cdots, z_r\}$ be the left coset representatives of $\mathbb{F}_{p^m}$ in $\mathbb{F}_{p^l}$. Thus
$$R_l(0) = R_m(0)+\sum_{j=1}^r R_m(z_j)=\sum_{j=1}^r \lambda(z_j)  =0.$$
Therefore we get  $\lambda(z_j+ u)=\lambda(z_j)$ and thus $\lambda(1+uz^{-1}_j)=1$ for any $u\in \mathbb{F}_{p^m}$ and $j=1,2,\cdots, r$. So it is not difficult to see that $\lambda(z)=1$ for any $z\in \mathbb{F}_{p^l} \setminus \mathbb{F}_{p^m}$, which contradicts the fact that $\lambda|_{ {\mathbb{F}}^*_{p^m}}$ is  nontrivial.
Therefore  Lemma \ref{keylemma} is proved.
\end{proof}

\section{The nontrivial character case}

In this section, we consider the case that $\theta$ is a nontrivial character.  Let $\varphi_e: \op{Ind}_{\bf T}^{\bf G} \Bbbk_\theta \rightarrow \mathbb{M}(\theta)$ be the natural morphism such that $\varphi_e ({\bf 1}_{\theta})= {\widehat{\bf 1}}_{\theta}$.  Let $\varphi_s: \op{Ind}_{\bf T}^{\bf G} \Bbbk_\theta \rightarrow \mathbb{M}(\theta^s) $ be the morphism such that $\varphi_s ({\bf 1}_{\theta})=s{\widehat{\bf 1}}_{\theta^s}$. Both $\mathbb{M}(\theta)$ and $\mathbb{M}(\theta^s)$ are the simple quotient modules of $\op{Ind}_{\bf T}^{\bf G}\Bbbk_{\theta}$. We will show that  the induced module $\op{Ind}_{\bf T}^{\bf G} \Bbbk_\theta$   has only two  simple  quotient modules  $\mathbb{M}(\theta)$ and $\mathbb{M}(\theta^s)$.

\begin{lemma} \label{nontrReplem}
Let $M$ be a $\Bbbk {\bf G}$-submodule of $ \op{Ind}_{\bf T} ^{\bf G} \Bbbk_{\theta}$. If there exists $n\in \mathbb{N}$ such that
$$\zeta=\sum_{z\in \mathbb{F}^*_{p^n} }  c_z   \overline{U_{p^n}} s  \varepsilon(z) {\bf 1}_{\theta} \in M,  \quad c_z\in \Bbbk,$$
where  $\displaystyle \sum_{z\in \mathbb{F}^*_{p^n} }  c_z  \ne 0$ and  $\displaystyle \sum_{z\in \mathbb{F}^*_{p^n} }  c_z \theta^s(-z) \ne 0$.
Then we have $\Bbbk {\bf G} \zeta=  \op{Ind}_{\bf T} ^{\bf G} \Bbbk_{\theta}$.
\end{lemma}

\begin{proof} Without loss of generality, we assume that $\theta |_{\mathbb{F}^*_{p^n}} \ne \theta^s|_{\mathbb{F}^*_{p^n}}$. Denote $C= \displaystyle \sum_{z\in \mathbb{F}^*_{p^n} }  c_z$ and $D= \displaystyle \sum_{z\in \mathbb{F}^*_{p^n} }  c_z \theta^s(-z)$.
By some easy computation,  we see that 
\begin{align}\label{eq4.1}
\overline{U_{p^n}} s \zeta= \theta(-1) C  \overline{U_{p^n} } {\bf 1}_{\theta}+D \overline{U_{p^n} }s {\bf 1}_{\theta} +   \sum_{z\in \mathbb{F}^*_{p^n} }  d_z   \overline{U_{p^n}} s  \varepsilon(z) {\bf 1}_{\theta}
\end{align}
where $d_z\in \Bbbk$.   Let $V$  be the space spanned by$\{  \overline{U_{p^n} } {\bf 1}_{\theta}, \overline{U_{p^n}} s  \varepsilon(z){\bf 1}_{\theta}\mid  z\in \mathbb{F}_{p^n}\}$.
Let $\omega$ be the generator of the cyclic group $\mathbb{F}^*_{p^n}$.  Denote by $\mathcal{A}:  V \rightarrow V$ the linear map introduced  by the action of $h(\omega^{\frac{1}{2}})$ on $V$. Let $\kappa =\theta^s(\omega^{-\frac{1}{2}})$. We see that $\kappa\ne \kappa^{-1}$ and $\kappa^{2(p^n-1)}=1$.  Then the minimal polynomial of $\mathcal{A}$ is $t^{p^n-1}-\kappa^{p^n-1}$.  When $r \nmid p^n-1$, then the linear map $\mathcal{A}$ is diagonalizable.
The eigenspace corresponding to eigenvalue $\kappa $  is spanned by  $\overline{U_{p^n}} s{\bf 1}_{\theta}$  and $\overline{U_{p^n}} s \overline{U^*_{p^n}}{\bf 1}_{\theta}$. Take the component of  $\zeta$ and $\overline{U_{p^n}} s \zeta$ in the eigenspace for eigenvalue 1. We see that  $\overline{U_{p^n}} s{\bf 1}_{\theta}\in \Bbbk {\bf G} \zeta$.  So using {\cite[Lemma 3.6]{CD2}},  we have ${\bf 1}_{\theta} \in  \Bbbk {\bf G}\zeta$ and thus $\Bbbk {\bf G}\zeta=\op{Ind}_{\bf T}^{\bf G} \Bbbk_\theta$.

 Now we consider the case $r| p^n-1$.  Write $p^n-1= ml$, where $m$ is a power of $r$ and $(l,r)=1$. Then the Jordan canonical form of $\mathcal{A}$ is $$J_1(\kappa^{-1}) \dot{+} J_1(\kappa) \dot{+} J_{m}(\kappa \lambda_1) \dot{+} J_{m}(\kappa\lambda_2) \dot{+} \cdots \dot{+} J_{m}(\kappa\lambda_l),$$
where $\lambda_1=1, \lambda_2, \cdots, \lambda_l$ are the roots of the equation $t^l-1=0$.  Let $\mathbf{u}_1, \mathbf{u}_2, \cdots, \mathbf{u}_m$ be a Jordan basis corresponding to the block $J_m(\lambda)$, where $(\mathcal{A}- \lambda \mathcal{E}) \mathbf{u}_{i+1}= \mathbf{u}_i$. Set $ \mathbf{u}'_i = \overline{U_{p^n}}  s  \mathbf{u}_i \in V$. Then
$$ \mathcal{A} \mathbf{u}'_{i+1}=  \overline{U_{p^n}}  s  \mathcal{A}^{-1} \mathbf{u}_{i+1} = \lambda^{-1}\overline{U_{p^n}}  s (  \mathbf{u}_{i+1} -  \mathcal{A}^{-1}\mathbf{u}_{i} )=\lambda^{-1}( \mathbf{u}'_{i+1}- \mathcal{A} \mathbf{u}'_{i}).$$
Thus  $(  \mathcal{A}- \lambda^{-1}  \mathcal{E})  \mathbf{u}'_{i+1}= - \lambda^{-1}\mathcal{A} \mathbf{u}'_{i}$.
Therefore we have 
\begin{align} \label{eq4.4}
\mathbf{u}'_{i}= \overline{U_{p^n}}  s  \mathbf{u}_i \in \ker{(\mathcal{A}- \lambda^{-1}  \mathcal{E})^{i} }
\end{align} by induction on $i$.  In particular, we see that
\begin{align}\label{eq4.5}
\overline{U_{p^n}} s:  \ker{(\mathcal{A}- \lambda  \mathcal{E})^{m}} \rightarrow  \ker{(\mathcal{A}- \lambda^{-1} \mathcal{E})^{m}}
\end{align}
is a linear operator by (\ref{eq4.4}).

Let  $\mathbf{v}_1, \mathbf{v}_2, \cdots, \mathbf{v}_m$ be a Jordan basis corresponding to the block $J_m(\kappa)$, where $(\mathcal{A}-\kappa \mathcal{E}) \mathbf{v}_{i+1}= \mathbf{v}_i$. We have $\mathbf{v}_1=\overline{U_{p^n}} s \overline{U^*_{p^n}}{\bf 1}_{\theta}$.   Let  $\mathbf{w}_1, \mathbf{w}_2, \cdots, \mathbf{w}_m$ be a Jordan basis corresponding to the block $J_m(\kappa^{-1})$, where $(\mathcal{A}-\kappa^{-1} \mathcal{E}) \mathbf{w}_{i+1}= \mathbf{w}_i$ and $\mathbf{w}_1=\displaystyle \overline{U_{p^n}} s \sum_{x\in  \mathbb{F}^*_{p^n}} \theta(x)\varepsilon(x){\bf 1}_{\theta}$. Let  $ \mathbf{v}'_i= \overline{U_{p^n}} s \mathbf{v}_i$ and $ \mathbf{w}'_i= \overline{U_{p^n}} s \mathbf{w}_i$. Denote $ \mathbf{x}=\overline{U_{p^n} } {\bf 1}_{\theta}$ and $ \mathbf{y}=\overline{U_{p^n} }s {\bf 1}_{\theta}$. Thus for any $k=1,2,\cdots, m$, we see that $\mathbf{v}'_k$
is a combination of  $ \mathbf{x}, \mathbf{w}_1, \mathbf{w}_2, \cdots, \mathbf{w}_k$ and  $\mathbf{w}'_k$
is a combination of  $ \mathbf{y}, \mathbf{v}_1, \mathbf{v}_2, \cdots, \mathbf{v}_k$.

Taking the direct summand  of  $\zeta $ in the Jordan block $J_m(\kappa) \dot{+} J_m(\kappa^{-1})$, we get an element
$$\zeta'= a_1 \mathbf{v}_1 + \cdots  + a_k\mathbf{v}_k +b_1 \mathbf{w}_1 + \cdots  + b_l\mathbf{w}_l\in \Bbbk {\bf G} \zeta, $$
where $a_k\ne 0$ and $b_l\ne 0$.  Thus we see that $ \mathbf{v}_k, \mathbf{w}_l\in \Bbbk {\bf G} \zeta$.
Take the direct summand  of  $\overline{U_{p^n}} s  \zeta $ in the Jordan block $J_1(\kappa^{-1}) \dot{+} J_1(\kappa) \dot{+} J_m(\kappa^{-1}) \dot{+} J_m(\kappa)$.  Using (\ref{eq4.1}) and  (\ref{eq4.5}),  we get an element
$$\zeta''=\theta(-1) C  \overline{U_{p^n} } {\bf 1}_{\theta}+D \overline{U_{p^n} }s {\bf 1}_{\theta} +\sum_{i=1}^k a'_i \mathbf{w}_i + \sum_{j=1}^l b'_j\mathbf{v}_j\in \Bbbk {\bf G} \zeta.$$
Thus we have
$$ \theta(-1) C  \overline{U_{p^n} } {\bf 1}_{\theta} +  \sum_{i=1}^k a'_i \mathbf{w}_i  \in \Bbbk {\bf G} \zeta, \quad  D \overline{U_{p^n} }s {\bf 1}_{\theta}+\sum_{j=1}^l b'_j\mathbf{v}_j   \in \Bbbk {\bf G} \zeta.$$
If $k\geq l$, since $ \mathbf{v}_k \in \Bbbk {\bf G} \zeta$,  we get $\overline{U_{p^n} }s {\bf 1}_{\theta}\in  \Bbbk {\bf G} \zeta$. If $k\leq l$, we see that $\overline{U_{p^n} }{\bf 1}_{\theta}\in  \Bbbk {\bf G} \zeta$ using $ \mathbf{w}_l \in \Bbbk {\bf G} \zeta$.  In both cases,  we have ${\bf 1}_{\theta} \in  \Bbbk {\bf G}\zeta$ by  {\cite[Lemma 3.6]{CD2}}  and thus $\Bbbk {\bf G}\zeta=\op{Ind}_{\bf T}^{\bf G} \Bbbk_\theta$. The lemma is proved.
\end{proof}

Let $\xi$ be an element such that  $\xi \notin  \ker \varphi_e$ and  $\xi \notin  \ker \varphi_s$. We will show that ${\Bbbk {\bf G}} \xi= \op{Ind}_{\bf T}^{\bf G}  \Bbbk_\theta  $. Thus $\op{Ind}_{\bf T}^{\bf G} \Bbbk_\theta$  has no other simple quotient modules except $\mathbb{M}(\theta)$ and $\mathbb{M}(\theta^s)$.  Denote
$$\xi= \sum_{z\in \bar{\mathbb{F}}_p}a_z \varepsilon(z){\bf 1}_{\theta}+ \sum_{x, y\in \bar{\mathbb{F}}_p} b_{x,y} \varepsilon(x)s\varepsilon(y){\bf 1}_{\theta}.$$
It is easy to see that there exists an element  $u\in  \bar{\mathbb{F}}_p$ such that $s \varepsilon(u) \xi$ has the following form
$$\eta= s \varepsilon(u) \xi= \sum_{x\in \bar{\mathbb{F}}_p,  y\in \bar{\mathbb{F}}^*_p } f_{x,y} \varepsilon(x)s\varepsilon(y){\bf 1}_{\theta} \in M.$$
Since $\xi \notin  \ker \varphi_e$ and  $\xi \notin  \ker \varphi_s$,  we also have $\eta \notin  \ker \varphi_e$ and  $\eta \notin  \ker \varphi_s$.
Denote $f_x =\displaystyle \sum_{ y\in \bar{\mathbb{F}}^*_p} f_{x,y}$.
Using $\eta \notin  \ker \varphi_e$, we see that $f_{x}\ne 0$ for some $x$.  Note that
\begin{align}\label{eq4.8}
\varphi_s(\eta) =  \sum_{x\in \bar{\mathbb{F}}_p,  y\in \bar{\mathbb{F}}^*_p } f_{x,y} \theta^s(-y)  \varepsilon(x-y^{-1}) s {\widehat{\bf 1}}_{\theta^s},
\end{align}
by some easy computation. 
For any $z  \in  \bar{\mathbb{F}}_p$, set  $ \Xi_z= \{(x, y)\mid x-y^{-1} =z\}.$ Since $\eta \notin  \ker \varphi_s$ and using (\ref{eq4.8}), we get
\begin{align}\label{eq4.2}
\sum_{(x,y) \in \Xi_{z} } f_{x,y} \theta^s(-y)\ne 0
\end{align}
for some $z\in \bar{\mathbb{F}}_p$.

Let $v\in  \bar{\mathbb{F}}^*_p$ such that $v+x \ne 0$ when $f_{x,y}\ne 0$ and
we consider
\begin{align}\label{eq4.3}
s \varepsilon(v)\eta= \sum_{x,y\in \bar{\mathbb{F}}_p} f_{x,y}\theta(-(v+x)) \varepsilon(-(v+x)^{-1})s \varepsilon((v+x)^2y-(v+x))
{{\bf 1}}_{\theta}.
\end{align}
There exists an integer $l$ such that for any $v\in  \bar{\mathbb{F}}_p \setminus \mathbb{F}_{p^l}$,   the elements $(v+x)^2y-(v+x)$ are nonzero and different from each other.  By Lemma \ref{keylemma}, we can choose some element $v_1 \in    \mathbb{F}_{p^n} \setminus  \mathbb{F}_{p^l}$ such that
$$f:=\sum_{x\in \bar{\mathbb{F}}_p} f_{x}\theta(h(-(v_1+x))) \ne 0.$$
Thus we get  
\begin{align}\label{eq4.6}
 \overline{U_{ p^n}} s \varepsilon(v_1)\eta =   \sum_{z\in \mathbb{F}^*_{p^n}} a_z \overline{U_{ p^n}} s  \varepsilon(z){{\bf 1}}_{\theta},  \quad \text{where} \ \sum_{z\in \mathbb{F}^*_{p^n}} a_z= f\ne 0.
\end{align}
In the expression (\ref{eq4.3}) of $s \varepsilon(v)\eta$, we consider
$$ \sum_{x,y\in \bar{\mathbb{F}}_p} f_{x,y} \theta(-(v+x)) \theta^s(-( (v+x)^2y-(v+x)))= \sum_{x,y\in \bar{\mathbb{F}}_p} f_{x,y} \theta^s( (v+x)y-1).$$
Using (\ref{eq4.2}), let $\sigma_1, \sigma_2, \dots, \sigma_r\in \bar{\mathbb{F}}_p$ such that  $f_i:= \displaystyle \sum_{(x,y) \in \Xi_{\sigma_i} } f_{x,y} \theta^s(-y)\ne 0$.
So we have $$ \sum_{x,y\in \bar{\mathbb{F}}_p} f_{x,y} \theta^s( (v+x)y-1)= \sum_{i=1}^r f_i \theta^s(-v-\sigma_i).$$
By Lemma \ref{keylemma}, we can choose some element $v_2 \in    \mathbb{F}_{p^n} \setminus  \mathbb{F}_{p^l}$ such that
$ \widetilde{f}:=\displaystyle \sum_{i=1}^r f_i \theta^s(-v_2-\sigma_i)\ne 0$. So we have
\begin{align}\label{eq4.7}
 \overline{U_{ p^n}} s \varepsilon(v_2)\eta =   \sum_{z\in \mathbb{F}_{p^n}} b_z \overline{U_{ p^n}} s  \varepsilon(z){{\bf 1}}_{\theta}, \quad where 
\  \sum_{z\in \mathbb{F}^*_{p^n}} b_z \theta^s(-z)= \widetilde{f} \ne 0.
\end{align}
Using (\ref{eq4.6}) and (\ref{eq4.7}),  there exists an element
$\displaystyle \zeta=   \sum_{z\in \mathbb{F}^*_{p^n}} c_z \overline{U_{ p^n}} s  \varepsilon(z){{\bf 1}}_{\theta} $
such that $\displaystyle  \sum_{z\in \mathbb{F}^*_{p^n}} c_z \ne 0$ and
 $\displaystyle  \sum_{z\in \mathbb{F}^*_{p^n}} c_z \theta^s(-z)\ne 0$. By Lemma \ref{nontrReplem}, we see that $\Bbbk {\bf G} \xi=\Bbbk {\bf G} \zeta=  \op{Ind}_{\bf T} ^{\bf G} \Bbbk_{\theta}$. Therefore the induced module $\op{Ind}_{\bf T}^{\bf G} \Bbbk_\theta$   has only two  simple  quotient modules  $\mathbb{M}(\theta)$ and $\mathbb{M}(\theta^s)$.

\section{Supplementary result}

Let  $G$ be a group, $\mathbb{F}$ be a field.  The characters $\sigma_1, \sigma_2, \cdots, \sigma_n$ of $G$ are called linearly dependent if there exists $a_1, a_2, \cdots, a_n \in \mathbb{F}$, not all zero such that
$$a_1\sigma_1(x)+ a_2 \sigma_2(x)+\cdots +a_n \sigma_n(x)=0, \quad  \forall x\in G.$$
If the characters are not  linearly  dependent, they are called  linearly  independent.
The following theorem on the  linearly  independence of characters are well-known.

\begin{theorem}  \cite[Theorem 12]{A}
  If $G$ is a group and $\sigma_1, \sigma_2, \cdots, \sigma_n$ are $n$ mutually distinct characters of $G$ in a field $\mathbb{F}$,  then $\sigma_1, \sigma_2, \cdots, \sigma_n$ are  linearly  independent.
\end{theorem}

Similar to the statement of this result,  we establish the linear  independence of the translates of  nontrivial  multiplicative characters over a finite field.
As in the introduction,  $\Bbbk$ is an  algebraically closed field with $\text{char} \ \Bbbk=r$. For a  nontrivial multiplicative character $\psi: {\mathbb{F}}^*_{p^m} \rightarrow \Bbbk^*$, we set  $\psi(0)=0$.

\begin{theorem} \label{LIND}
Suppose that $r\ne p$.  Let $\psi: {\mathbb{F}}^*_{p^m} \rightarrow \Bbbk^*$ be a nontrivial multiplicative character. Given different elements  $u_1, u_2, \dots, u_n\in {\mathbb{F}}^*_{p^m}$,  let $f_i: \mathbb{F}_{p^m} \rightarrow \Bbbk $ be a map defined by $f_{i}(x)= \psi(x+u_i)$ for any $x\in \mathbb{F}_{p^m}$. Then $f_1, f_2, \cdots, f_n$ are linearly independent.
\end{theorem}

\begin{proof}
Since $\psi$ is a nontrivial multiplicative character of  $\mathbb{F}_{p^m}$, we see that
  $\displaystyle \sum_{z\in  {\mathbb{F}}_{p^m}}\psi(z)=0$ by Lemma \ref{Orth}.  Thus for any $u\in {\mathbb{F}}^*_{p^m}$, we get
  $$ \sum_{x\in  {\mathbb{F}}^*_{p^m}}\psi(x+u) \psi(x^{-1})= \sum_{x\in  {\mathbb{F}}^*_{p^m}}\psi(1+ux^{-1})=-1.$$
Now  we consider the equation  $\displaystyle \sum_{i=1}^n a_i f_i=0$, where $a_i\in \Bbbk$.  We have
 $$\sum_{i=1}^n\sum_{x\in  \mathbb{F}^*_{p^m} } a_i f_i (x) \psi(x^{-1})=-\sum_{i=1}^n a_i=0.$$
 On the other hand, for any fixed integer $k=1, 2, \cdots, n$, we see that
$$\sum_{i=1}^n\sum_{x\in  \mathbb{F}_{p^m}\setminus  \{-u_k\} } a_i f_i (x) \psi((x+u_k)^{-1})=(p^m-1)a_k - \sum_{i\ne k} a_i=p^m a_k.$$
Noting $\text{char} \ \Bbbk  \ne p$, we get  $a_k=0$ for any  $k=1, 2, \cdots, n$. So  $f_1, f_2, \cdots, f_n$ are linearly independent.
\end{proof}

\begin{corollary} \label{Cor}
Suppose that $r\ne p$.  Let $\psi: \bar{\mathbb{F}}^*_{p} \rightarrow \Bbbk^*$ be a nontrivial multiplicative character.
 Given different elements  $u_1, u_2, \dots, u_n\in \bar{\mathbb{F}}^*_{p}$,  let $f_i: \bar{\mathbb{F}}_{p} \rightarrow \Bbbk $ be a map defined by $f_{i}(x)= \psi(x+u_i)$ for any $x\in \bar{\mathbb{F}}_{p}$. Then $f_1, f_2, \cdots, f_n$ are linearly independent.
\end{corollary}

\begin{proof}
There exists $m\in \mathbb{N}$ such that  $u_1, u_2, \cdots, u_n\in \mathbb{F}^*_{p^m}$ and $\psi| \mathbb{F}^*_{p^m}$ is nontrivial. Using Theorem \ref{LIND}, the corollary is proved.
\end{proof}

\begin{remark} \normalfont
Theorem  \ref{LIND} and Corollary \ref{Cor} do not hold if  $\Bbbk=\bar{\mathbb{F}}_p$.  Indeed,  let  $q$ be a power of $p$ and  $u_1, u_2, \dots, u_q\in \bar{\mathbb{F}}^*_p$,  which are different from each other and satisfy
$\displaystyle \sum_{i=1}^q u^q_i =0.$  Let $\psi : \bar{\mathbb{F}}^*_p \rightarrow \bar{\mathbb{F}}^*_p$ be the character given by $\psi(x)=x^q$ for any $x\in \bar{\mathbb{F}}^*_p$. Then for any $x\in  \bar{\mathbb{F}}^*_p$,  we have
$$\sum_{i= 1}^q f_i (x)=\sum_{i= 1}^q   (x+u_i)^q=\sum_{i= 1}^q   (x^q+ u^q_i)=0.$$
 Thus $f_1, f_2, \cdots, f_n$ are linearly dependent.
\end{remark}

\bigskip

\noindent{\bf Acknowledgements} The author is grateful to  Nanhua Xi  for his consistent encouragement throughout the research.

\medskip

\noindent{\bf Statements and Declarations}  The author declares that he has no conflict of interests with others.

\medskip

\noindent{\bf Data Availability}  Data sharing not applicable to this article as no datasets were generated or analysed during the current study.

\bigskip

\bibliographystyle{amsplain}

\begin{thebibliography}{10}


\bibitem {A}

E. Artin,   \textit{Galois Theory},   2nd edition, Notre Dame Math. Lectures, vol. 2, University of Notre Dame, Notre Dame, 1944.


\bibitem {BT}
A. Borel and J. Tits,  \textit{Homomorphismes "abstraits" de groupes algebriques simples},  Ann. of Math. (2) {97} (1973), 99--571.



\bibitem {CD2}
X. Chen, J. Dong, \textit{Abstract-induced modules for reductive algebraic groups with Frobenius maps},  Int. Math. Res. Not. IMRN {2022}, no. 5, 3308--3348.


\bibitem {D}
J.  Dong, \textit{Certain complex representations of  $SL_2(\overline{\mathbb{F}}_q)$}, J. Algebra {612} (2022), 504--525.

\bibitem {D1}
J.  Dong,   \textit{Some conjectures on the quotients of the tensor products   in the category $\mathscr{X}$}, arXiv:2508.00488, accepted by J. Lie Theory.


\bibitem {ESS}
J. H. Evertse,  H.P. Schlickewei and W.M. Schmidt, \textit{Linear equations in variables which lie in a multiplicative group}, Ann. of Math. (2) {155} (2002), no.3, 807--836.

\bibitem {M}

V. Mazorchuk,  \textit{Lectures  on $\mathfrak{sl}_2(\mathbb{C})$-modules}, Imperial College Press, London, 2010.

\bibitem {Xi}
N. Xi,  \textit{Some infinite dimensional representations of reductive groups with Frobenius maps}, Sci. China Math.  57  (2014), 1109--1120.







\end{thebibliography}

\end{document}